\documentclass[11pt]{amsart}
\usepackage{amsmath}
\usepackage{verbatim}
\usepackage{amsfonts}
\usepackage{amssymb}
\usepackage{amsthm}
\usepackage{mathrsfs} 
\usepackage{tikz-cd}
\usepackage{color}
\usepackage[all]{xy}  
\usepackage{enumerate}
\usepackage[margin=1in]{geometry}
\usepackage{mathrsfs}
\usepackage{stmaryrd}
\usepackage{bbold}
\usepackage{microtype}
\usepackage[bookmarks=false,colorlinks=true,hyperindex,citecolor=green,linkcolor=red]{hyperref}

\theoremstyle{definition}
\numberwithin{equation}{section}

\newtheorem{lem}{Lemma}[section]

\newtheorem{defn}{Definition}[section]

\newtheorem{prop}{Proposition}[section]

\newtheorem*{thm}{Theorem}

\newtheorem{theor}{Theorem}[section]

\newtheorem{corol}{Corollary}[section]

\newenvironment{changemargin}[2]{
  \begin{list}{}{
      \setlength{\topsep}{0cm}
      \setlength{\leftmargin}{#1}
      \setlength{\rightmargin}{#2}
    }
  \item[]
}{
  \end{list}
}

\newcommand{\mf}[1]{\mathfrak{#1}}

\newcommand{\md}[1]{#1\raise0.7ex\hbox{$\scriptstyle\vee$}}
\newcommand{\mda}[2]{#1\raise0.7ex\hbox{$\scriptstyle\vee_{#2}$}}

\title[The local cohomology of a parameter ideal]{The local cohomology of a parameter ideal with respect to an arbitrary ideal}
  
\author[Monica Lewis]{Monica Ann Lewis}
\address{Department of Mathematics, University of Michigan, Ann Arbor, MI 48109-1043, USA}
\email{malewi@umich.edu}


\begin{document}

\vspace{-1.5em}
\begin{abstract}
  Let $R$ be a regular ring, let $J$ be an ideal generated by a regular sequence of codimension at least $2$, and let $I$ be an ideal containing $J$.
  We give an example of a module $H^3_I(J)$ with infinitely many associated primes, answering a question of Hochster and N\'u\~nez-Betancourt in the negative. In fact, for $i\leq 4$, we show that under suitable hypotheses on $R/J$, $\text{Ass}\,H^{i}_I(J)$ is finite if and only if $\text{Ass}\,H^{i-1}_I(R/J)$ is finite. Our proof of this statement involves a novel generalization of an isomorphism of Hellus, which may be of some independent interest. The finiteness comparison between $\text{Ass}\, H^i_I(J)$ and $\text{Ass}\, H^{i-1}_I(R/J)$ tends to improve as our hypotheses on $R/J$ become more restrictive. To illustrate the extreme end of this phenomenon, at least in the prime characteristic $p>0$ setting, we show that if $R/J$ is regular, then $\text{Ass}\, H^i_I(J)$ is finite for all $i\geq 0$.
\end{abstract}
\maketitle

\section{Introduction}
Local cohomology modules over a regular ring are known in many cases to exhibit remarkable finiteness properties. For a regular ring $S$, one may ask whether the set $\text{Ass}\, H^i_I(S)$ is finite for any ideal $I\subseteq S$ and any $i\geq 0$. A celebrated theorem originally due to Huneke and Sharp (later generalized by Lyubeznik's theory of $F$-modules \cite{lyufmod}) says that this is indeed the case if $S$ has prime characteristic $p>0$ \cite[Corollary 2.3]{hush}. Lyubeznik proved the corresponding statement for smooth $K$-algebras when $K$ is a field of characteristic $0$ \cite[Remark 3.7(i)]{lyudmod} and for any regular local ring containing $\mathbb{Q}$ \cite[Theorem 3.4]{lyudmod}. Concerning regular rings of mixed characteristic, the property is known to hold when $S$ is an unramified regular local ring \cite[Theorem 1]{lyuunram}, a smooth $\mathbb{Z}$-algebra \cite[Theorem 1.2]{bhatt}, or is local and of dimension $\leq 4$ \cite[Theorem 2.9]{marley}.

The finiteness properties of local cohomology modules $H^i_I(S)$ when $S$ is a complete intersection ring are far less well-understood than when $S$ is regular. Infinite sets of associated primes can be found even when $S$ is a hypersurface ring. For example, Singh describes a hypersurface ring $S$ finitely generated over $\mathbb{Z}$,
$$
S = \frac{\mathbb{Z}[u,v,w,x,y,z]}{ux+vy+wz}
$$
such that for all prime integers $p$, the module $H^3_{(x,y,z)}(S)$ has a nonzero $p$-torsion element \cite{singhp}. This is not just a global phenomenon: Katzman later gave a local hypersurface ring $S$ containing a field $K$,
$$
S = \frac{K[[u,v,w,x,y,z]]}{wu^2x^2-(w+z)uxvy+zv^2y^2}
$$
such that $\text{Ass}\, H^2_{(x,y)}(S)$ is infinite \cite{katzinf}. Singh and Swanson construct examples of equicharacteristic local hypersurface rings to demonstrate that $\text{Ass}\, H^3_I(S)$ can be infinite even if $S$ is a UFD, an $F$-regular ring of characteristic $p>0$, or a characteristic $0$ ring with rational singularities \cite[Theorem 5.4]{singhswan}.

Surprisingly, the local cohomology of a hypersurface ring is still known to possess striking finiteness properties, at least in the characteristic $p>0$ setting. A result proved independently by Katzman and Zhang \cite[Theorem 7.1]{katzhang} or by Hochster and N\'{u}\~{n}ez-Betancourt \cite[Corollary 4.13]{hochnb} says that if $R$ is a regular ring of prime characteristic $p>0$, $f\in R$ is a nonzerodivisor, and $S=R/f$, then for any ideal $I$ and any $i\geq 0$, $H^i_I(S)$ has only finitely many minimal primes. Equivalently, $\text{Supp}\, H^i_I(S)$ is a Zariski closed set.

There are other situations in which the closed support property holds. For example, if $M$ finitely generated over $S$, then the set $\text{Supp}\,H^i_I(M)$ is known to be closed (i) when $S$ has prime characteristic $p>0$, $I$ is generated by $i$ elements, and $M=S$ \cite[Theorem 2.10]{katztop}, (ii) when $S$ is standard graded, $M$ is graded, $I$ is the irrelevant ideal, and $i$ is the cohomological dimension of $I$ on $M$ \cite[Theorem 1]{rottsega}, (iii) when $S$ is local of dimension at most $4$ \cite[Proposition 3.4]{hunkatmar}, or (iv) when $I$ has cohomological dimension at most $2$ \cite[Theorem 2.4]{hunkatmar}. It is not presently known how far these results generalize. Hochster asks whether the support of $H^i_I(M)$ must be closed for all $i\geq 0$ and all ideals $I$, where $S$ is assumed only to be Noetherian, and $M$ may be any finitely generated $S$-module \cite[Question 2]{hochreview}. Even at this level of generality, the question remains open.

The present paper is motivated by Hochster's question in the complete intersection setting. Namely, does the closed support property of characteristic $p>0$ hypersurface rings \cite{katzhang,hochnb} generalize to complete intersection rings of higher codimension?

\vspace{0.5em}
\noindent \textbf{Question 1a.} Let $S$ be a complete intersection ring, let $I$ be an arbitrary ideal of $S$, and fix $i\geq 0$. Is $\text{Supp}\, H^i_I(S)$ a Zariski closed set?
\vspace{0.5em}

For simplicity, we restrict our focus to when $S$ is given by a presentation $S=R/J$ for $R$ a regular ring and $J$ an ideal generated by a regular sequence in $R$. Not all regular rings $R$ are known to have the property that $\text{Ass}\, H^i_I(R)$ is always finite. In order to avoid potential complications arising from $R$, we impose an additional hypothesis.

\begin{defn}
Call a Noetherian ring $R$ {\em LC-finite} if, for any ideal $I$ and any $i\geq 0$, the module $H^i_I(R)$ has a finite set of associated primes.
\end{defn}

We do not assume regularity in the definition. A semilocal ring of dimension at most $1$ is trivially LC-finite, but can easily fail to be regular. Likewise, any $F$-finite ring with finite $F$-representation type (FFRT) is LC-finite \cite[Theorem 5.7]{hochnb}, but need not be regular.
The class of LC-finite regular rings includes, for example, all regular local rings of dimension $\leq 4$, all regular rings of prime characteristic $p$, regular local rings containing $\mathbb{Q}$, smooth $K$-algebras for $K$ a field of characteristic $0$, unramified regular local rings of mixed characteristic, and smooth $\mathbb{Z}$-algebras. The class of LC-finite rings is closed under localization. If there is a finite set of maximal ideals $\mathfrak{m}_1,\cdots,\mathfrak{m}_t$ of $R$ such that $\text{Spec}(R)-\{\mathfrak{m}_1,\cdots,\mathfrak{m}_t\}$ can be covered by finitely many charts $\text{Spec}(R_f)$, each of which is LC-finite, then $R$ is LC-finite. For example, a ring of prime characteristic $p$ with isolated singularities is LC-finite. If $R$ is LC-finite and $A\to R$ is pure (e.g., if $A$ is a direct summand of $R$), then $A$ is LC-finite \cite[Theorem 3.1(d)]{hochnb}. We narrow the scope of Question 1a as follows.

\vspace{0.5em}
\noindent \textbf{Question 1b.} Let $R$ be an LC-finite regular ring, $J\subseteq R$ be an ideal
generated by a regular sequence, and $S=R/J$. Let $I$ be an arbitrary ideal of $S$, and fix $i\geq 0$. Is $\text{Supp}\, H^i_I(S)$ a Zariski closed set?
\vspace{0.5em}

If $R$ is a regular ring of prime characteristic $p>0$, Hochster and N\'{u}\~{n}ez-Betancourt show that if $\text{Ass}\, H^{i+1}_I(J)$ is finite, then $\text{Supp}\, H^i_I(R/J)$ is closed \cite[Theorem 4.12]{hochnb}. Their theorem raises an immediate question. Although (to the best of our knowledge) it is not yet known whether Hochster and N\'{u}\~{n}ez-Betancourt's result generalizes to (equal or mixed) characteristic $0$, we will nonetheless investigate the following question for LC-finite regular rings of any characteristic.

\vspace{0.5em}
\noindent \textbf{Question 2.}
Let $R$ be an LC-finite regular ring, $J\subseteq R$ be an ideal
generated by a regular sequence, $I$ be an ideal of $R$ containing $J$ (corresponding to an arbitrary ideal of $R/J$), and fix $i\geq 0$. Is $\text{Ass}\, H^i_I(J)$ a finite set?
\vspace{0.5em}

 In Section 2, using some properties of the ideal transform functor associated with $I$, we show that, even under weaker hypotheses on $R$ and $J$, this question has a positive answer when $i=2$. These hypotheses include the case where $J=R$ and $R$ is a $\mathbb{Q}$-factorial normal ring (cf. Dao and Quy \cite[Theorem 3.3]{daoquyH2}).
\begin{thm}[\ref{H2}]
  Let $R$ be a \textit{locally almost factorial} (see Definition \ref{laf}) Noetherian normal ring, and let $I$, $J$ be ideals of $R$. The set $\text{Ass}\, H^2_I(J)$ is finite.
\end{thm}
The result does not generalize to $i>2$. In Section 3, we give an example where $H^3_I(J)$ has infinitely many associated primes. In this example, $R$ is a $7$-dimensional polynomial ring, $J$ is generated by a regular sequence of length $2$, and $I$ is $4$-generated. Question 2 has a negative answer at the level of generality in which it is stated. However, our counterexample crucially requires $\text{Ass}\, H^2_I(R/J)$ to be infinite. In fact, $R/J$ in our counterexample is isomorphic to Katzman's hypersurface \cite{katzinf}. The nature of the counterexample therefore begs the following natural question that, to the best of our knowledge, remains open.

\vspace{0.5em}
\noindent \textbf{Question 3.}
Let $R$ be an LC-finite regular ring, $J\subseteq R$ be an ideal generated by a regular sequence, $I$ be an ideal of $R$ containing $J$, and fix $i\geq 0$. Does the finiteness of $\text{Ass}\, H^{i-1}_I(R/J)$ imply the finiteness of $\text{Ass}\, H^i_I(J)$?
\vspace{0.5em}

Notice that if $\text{depth}_J(R)=1$, then $J\simeq R$ as an $R$-module, so the question has a trivially positive answer. The question is only interesting when $\text{depth}_J(R)\geq 2$, and our methods in fact require $\text{depth}_J(R)\geq 2$.

We investigate this question in Section 4 for various small values of $i$. We give a partial positive answer to both this question and its converse when $i\leq 4$. We emphasize in this result that $R$ is allowed to have any characteristic, even mixed characteristic, so long as it satisfies the LC-finiteness condition.

\begin{thm}[\ref{smalli}]
  Let $R$ be an LC-finite regular ring, let $J\subseteq R$ be an ideal generated by a regular sequence of length $j\geq 2$, and let $S=R/J$. For $I$ an ideal of $R$ containing $J$,
  \begin{enumerate}
  \item[(i)] $\text{Ass}\, H^i_I(J)$ and $\text{Ass}\, H^{i-1}_I(S)$ are always finite for $i\leq 2$.
  \item[(ii)] If the irreducible components of $\text{Spec}(S)$ are disjoint (e.g. $S$ is a domain), then $\text{Ass}\, H^{3}_I(J)$ is finite if and only if $\text{Ass}\,{H^2_I(S)}$ is finite.
  \item[(iii)] If $S$ is normal and locally almost factorial (e.g. $S$ is a UFD), then $\text{Ass}\,{H^4_I(J)}$ is finite if and only if $\text{Ass}\, H^3_I(S)$ is finite.
  \end{enumerate}
\end{thm}

A key aspect of Hochster and Núñez-Betancourt's strategy in codimension $1$ was using the finiteness of $\text{Ass}\, H^i_I(J)$ to control $\text{Min}\, H^{i-1}_I(S)$, even in a situation where $\text{Ass}\, H^{i-1}_I(S)$ may be infinite. However, the ``only if'' statements in Theorem \ref{smalli} imply that this strategy cannot be straightforwardly generalized to the higher codimension setting. Namely, if $\text{Ass}\, H^{i-1}_I(S)$ is infinite, then the same may necessarily also true of $\text{Ass}\, H^{i}_I(J)$.

The proof of Theorem \ref{smalli} relies heavily on the following isomorphism of functors, which may be of some independent interest. This isomorphism is a generalization of a result of Hellus \cite[Theorem 3]{hellus}.
\begin{thm}[\ref{genhel}]
Let $R$ be a Noetherian ring and let $I\subseteq R$ be any ideal. Fix $i\geq 0$. There is an ideal $I^\prime\supseteq I$ (resp. $I^{\prime\prime}\supseteq I$) such that
    \begin{itemize}
    \item There is a natural isomorphism $H^i_{I^\prime}(-) \xrightarrow{\sim} H^i_{I}(-)$ (resp. there is a natural surjection $H^i_{I^{\prime\prime}}(-) \twoheadrightarrow H^i_{I}(-)$)
    \item $\text{ht}(I^\prime)\geq i-1$ (resp. $\text{ht}(I^{\prime\prime})\geq i$)
    \end{itemize}
\end{thm}

Separate to our main application, Theorem \ref{genhel} yields a new proof of a result of Marley \cite[Proposition 2.3]{marley} on the set of height $i$ primes in the support of modules of the form $H^i_I(M)$. See Corollary \ref{heightcontrol}. 

Returning to the main objects of study, we observe that in Theorem \ref{smalli}, our control over the associated primes of $H^i_I(J)$ tends to become greater as our hypotheses on $R/J$ become more restrictive. We turn in Section 5 to investigating the extreme case in which $R/J$ is itself regular and LC-finite. At least in characteristic $p>0$, we obtain the following result.

\begin{thm}[\ref{charpj}]
  Let $R$ be a regular ring of prime characteristic $p>0$, let $J\subseteq R$ be an ideal such that $R/J$ is regular, let $\mathcal{M}$ be an $F$-finite $F_R$-module (e.g. $\mathcal{M}=R$), and let $I$ be any ideal of $R$. For all $i\geq 0$, $\text{Ass}\, H^i_I(J\mathcal{M})$ is finite.
\end{thm}

\noindent\textit{Convention:} Throughout this paper, we assume that all given rings are Noetherian, unless stated otherwise.

\section{Finiteness of $\text{Ass}\, H^2_I(J)$}

\noindent If we are focused exclusively on cohomological degree $2$, many of the hypotheses of our basic setting can be relaxed. $R$ need only be normal with a condition somewhat weaker than local factoriality, and the ideal $J\subseteq R$ can be completely arbitrary -- we do not need to $J$ to be generated by a regular sequence. Our goal is to show that $\text{Ass}\, H^2_I(J)$ is finite for any ideal $I\subseteq R$. The main case is when $\text{depth}_I(R)=1$.

\begin{lem}\label{depth1}
  Let $R$ be a Noetherian domain, and let $J\subseteq R$ be an ideal. If $I\subseteq R$ is an ideal such that $\text{depth}_I(R)\neq 1$, then $\text{Ass}\, H^2_I(J)$ is finite.
\end{lem}
\begin{proof}
If $I= (0)$ or $I=R$, there is nothing to do, so we assume that $I$ is a nonzero proper ideal. Since $R$ is a domain, this implies that both $J$ and $R$ are $I$-torsionfree, giving $\text{depth}_I(R)>0$ and by hypothesis $\text{depth}_I(R)\neq 1$, so we have $\text{depth}_I(R)\geq 2$. The following sequence is exact.
\begin{center}{
\begin{tikzpicture}[descr/.style={fill=white,inner sep=1.5pt}]
        \matrix (m) [
            matrix of math nodes,
            row sep=1em,
            column sep=2.5em,
            text height=1.5ex, text depth=0.25ex
        ]
        { 0 & 0 & 0 & \Gamma_I(R/J) \\
            & H^1_I(J) & 0 & H^1_I(R/J) \\
          & H^{2}_I(J) & H^{2}_I(R) & H^{2}_I(R/J) \\
        };

        \path[overlay,->, font=\scriptsize,>=latex]
        (m-1-1) edge (m-1-2)
        (m-1-2) edge (m-1-3)
        (m-1-3) edge (m-1-4)
        (m-1-4) edge[out=345,in=165] (m-2-2)
        (m-2-2) edge (m-2-3)
        (m-2-3) edge (m-2-4)
        (m-2-4) edge[out=345,in=165] (m-3-2)
        (m-3-2) edge (m-3-3)
        (m-3-3) edge (m-3-4)
;
\end{tikzpicture}}
\end{center}
Note that $H^1_I(J) \simeq \Gamma_I(R/J)$ is finitely generated, meaning that $H^2_I(J)$ is either finitely generated or the first non-finitely-generated local cohomology module of $J$ on $I$, and the stated result follows at once from Brodmann and Lashgari Faghani \cite[Theorem 2.2]{brodlash}.
\end{proof}

\begin{lem}\label{facdecomp}
Let $R$ be a Noetherian normal ring, and $I\subseteq R$ be an ideal such that $\text{depth}_I(R)=1$. Then $\sqrt{I}=L\cap I_0$ for some ideal $L$ given by the intersection of height one primes, and some ideal $I_0\subseteq R$ with $\text{depth}_{I_0}(R)\geq 2$.
\end{lem}
\begin{proof}
  In a Noetherian normal ring $R$, $\text{ht}(I_0)\geq 2$ implies $\text{depth}_{I_0}(R)\geq 2$.
\end{proof}

Recall the notion of an almost factorial ring. 
\begin{defn}\label{laf}
  A normal domain $R$ is called \textit{almost factorial} if the class group of $R$
  is torsion. A normal ring $R$ is called \textit{locally almost factorial} if $R_P$ is almost factorial for all $P\in\text{Spec}(R)$.
\end{defn}
A regular ring, for example, is locally (almost) factorial. Hellus shows that an almost factorial Cohen-Macaulay local ring of dimension at most four is LC-finite \cite[Theorem 5]{hellus}. Our application of the almost factorial hypothesis is motivated by the application in \cite{hellus}.

 The main property we require is that every height $1$ prime of an almost factorial ring is principal up to taking radicals. Recall that an ideal $I$ is said to have \text{pure} height $h$ if every minimal prime of $I$ has height exactly $h$. An ideal of pure height $1$ is (up to radical) a product of height $1$ primes, so in an almost factorial ring, any pure height $1$ ideal is principal up to radicals.

\begin{lem}
Let $R$ be a locally almost factorial Noetherian normal ring, and $L$ be an ideal of pure height $1$. Then there is a finite cover of $\text{Spec}(R)$ by open charts $\text{Spec}(R_{f_1}),\cdots,\text{Spec}(R_{f_t})$ such that for each $i$, $LR_{f_i}$ has the same radical as a principal ideal.
\end{lem}
\begin{proof}
  We do no harm in replacing $L$ with $\sqrt{L}$, so assume $L$ is radical. Consider a single point $P\in\text{Spec}(R)$. Since $R_P$ is almost factorial, we can write $LR_P = \sqrt{y R_P}$ for some $y\in R_P$. Up to multiplying by units of $R_P$, we may assume that $y$ is an element of $R$. Since
  $y\in LR_P\cap R$, there is some $u\in R-P$ such that $uy\in L$. Also, since $R$ is Noetherian,
  there is some $n>0$ such that $L^n R_P\subseteq yR_P$, hence $L^n\subseteq yR_P\cap R$, and there
  is some $v\in R-P$ such that $vL^n\subseteq yR$. If $f=uv$, then we see that
  $y\in LR_f$ and $L^n\subseteq yR_f$, giving $LR_f =\sqrt{yR_f}$.

  Our choice of $f$ depends on $P$. Varying over all $P\in\text{Spec}(R)$, we obtain a collection of open charts $\{\text{Spec}(R_{f_P})\}_{P\in\text{Spec}(R)}$ which cover $\text{Spec}(R)$ such that (the expansion of) $L$ is principal up to radicals on each chart. Since $\text{Spec}(R)$ is quasicompact, finitely many of these charts cover the whole space.
\end{proof}

\begin{corol}\label{localprinc}
  Let $R$ be a locally almost factorial Noetherian normal ring, and $I\subseteq R$ be an ideal such that $\text{depth}_I(R)=1$. Then there is an ideal $I_0\subseteq R$ with $\text{depth}_{I_0}(R)\geq 2$, and a finite cover of $\text{Spec}(R)$ by open charts $\text{Spec}(R_{f_1}),\cdots,\text{Spec}(R_{f_t})$ such that for each $i$,  $\sqrt{IR_{f_i}}=\sqrt{y_iR_{f_i}}\cap I_0$ for some $y_i\in R$.
\end{corol}

To show that $\text{Ass}\, H^2_I(J)$ is finite, it would certainly suffice to show that $\text{Ass}\, H^2_{IR_{f}}(JR_{f})$ is finite on each chart $\text{Spec}(R_f)$ of a finite cover of $\text{Spec}(R)$. Thus, up to radicals, the main case to understand is when $I$ has the form $yR \cap I_0$, with $y\in R$ and $\text{depth}_{I_0}(R)\geq 2$. This decomposition can be used to get a better understanding of the $I$-transform functor.

For an ideal $\mathfrak{a}\subseteq R$, recall the $\mathfrak{a}$-transform $$D_{\mathfrak{a}}(-) := \varinjlim_{t} \text{Hom}_R(\mathfrak{a}^t,-)$$ is a left exact covariant functor whose right derived functors satisfy $\mathscr{R}^i D_{\mathfrak{a}}(-) \simeq H^{i+1}_{\mathfrak{a}}(-)$. There is a sense in which $D_{\mathfrak{a}}(-)$ forces $\text{depth}_{\mathfrak{a}}(-)\geq 2$ without modifying higher local cohomology on $\mathfrak{a}$. Namely, for any $R$-module $M$, $\Gamma_{\mathfrak{a}}(D_{\mathfrak{a}}(M))=H^1_{\mathfrak{a}}(D_{\mathfrak{a}}(M))=0$, and $H^i_{\mathfrak{a}}(D_{\mathfrak{a}}(M))=H^i_{\mathfrak{a}}(M)$ for all $i\geq 2$. Below, we collect a few fundamental facts about the ideal transform, all of which can be found in the treatment presented in Chapter 2 of Brodmann and Sharp \cite{brodsharp}, along with proofs of the previous assertions.

\textit{Notational remark:} If $F$ and $G$ are functors $\mathfrak{C}\to\mathfrak{D}$, we will write natural
transformations from $F$ to $G$ as $\phi(-):F(-)\to G(-)$, which consists of the data of a map
denoted $\phi(A):F(A)\to G(A)$ for each object $A$ of $\mathfrak{C}$, such that for
any $\mathfrak{C}$-morphism $f:A\to B$, $\phi(B)\circ F(f) = G(f)\circ \phi(A)$.

\begin{lem}\label{four}
  (Brodmann and Sharp \cite[Theorem 2.2.4(i)]{brodsharp}) Let $R$ be a Noetherian ring and fix an ideal ${\mathfrak{a}}\subseteq R$. There is a natural transformation $\eta_{\mathfrak{a}}(-):\text{Id}\to D_{\mathfrak{a}}(-)$ such that, for any $R$-module $M$, there is an exact sequence
    $$
    0\to\Gamma_{\mathfrak{a}}(M)\to M\xrightarrow{\eta_{\mathfrak{a}}(M)} D_{\mathfrak{a}}(M)\to H^1_{\mathfrak{a}}(M)\to 0
    $$
\end{lem}

\begin{lem}\label{fac}
  (Brodmann and Sharp \cite[Proposition 2.2.13]{brodsharp}) Let $R$ be a Noetherian ring, and ${\mathfrak{a}}\subseteq R$ be an ideal. Let $e:M\to M^\prime$ be a homomorphism of $R$-modules such that $\text{Ker}{e}$ and $\text{Coker}{e}$ are both ${\mathfrak{a}}$-torsion. Then
  \begin{enumerate}[(i)]
  \item The map $D_{\mathfrak{a}}(e):D_{\mathfrak{a}}(M)\to D_{\mathfrak{a}}(M^\prime)$ is an isomorphism.
  \item There is a unique $R$-module homomorphism $\varphi:M^\prime\to D_{\mathfrak{a}}(M)$ such that the diagram
    \begin{center}
    \begin{tikzcd}
      M\ar[r,"e"]\ar[dr,"\eta_{\mathfrak{a}}(M)"'] & M^\prime\ar[d,"\varphi"]\\
      &D_{\mathfrak{a}}(M)
    \end{tikzcd}
    \end{center}
    commutes. In fact, $\varphi = D_{\mathfrak{a}}(e)^{-1}\circ \eta_{\mathfrak{a}}(M^\prime)$.
  \item The map $\varphi$ of (ii) is an isomorphism if and only if $\eta_{\mathfrak{a}}(M^\prime)$ is an isomorphism, and this is the case if and only if $\Gamma_{\mathfrak{a}}(M^\prime)=H^1_{\mathfrak{a}}(M^\prime)=0$.
  \end{enumerate}
\end{lem}

The main ingredient to dealing with $H^2_I(J)$ when $\text{depth}_I(R)=1$ is the following compatibility property of the $I$-transform functor with our decomposition $I=yR\cap I_0$.
\begin{prop}\label{loc}
Let $R$ be a Noetherian ring, $y$ an element of $R$, and $I_0\subseteq R$ an ideal. Let $I=yR\cap I_0$. There is a natural isomorphism of functors $D_{I_0}(-)_y\simeq D_{I}(-)$
\end{prop}
\begin{proof}
Precomposing $\eta_{I_0}(-)_y:(-)_y\to D_{I_0}(-)_y$ with $\text{Id}\to (-)_y$ we obtain a natural transformation $\gamma(-):\text{Id}\to D_{I_0}(-)_y$. We claim that for any module $M$, both the kernel and cokernel of $\gamma(M):M\to D_{I_0}(M)_y$ are $I=yR\cap I_0$-torsion:
    \begin{itemize}
    \item $\text{Ker}{\gamma(M)}$ consists of those $m\in M$ such that $m/1\in \Gamma_{I_0}(M)_y$, or, equivalently, $y^t m\in \Gamma_{I_0}(M)$ for some $t\geq 0$. Let $s$ be such that $I_0^s y^t m =0$. Then
      $(yI_0)^{\max(s,t)}m=0$, so $m\in \Gamma_{yI_0}(M)=\Gamma_{yR\cap I_0}(M)$ (since $\sqrt{yI_0} = \sqrt{yR\cap I_0}$).
    \item An element of $C=\text{Coker}{\gamma(M)}$ can be represented by $c=f/y^t$ for some $f\in D_{I_0}(M)$, $t\geq 0$. $\text{Coker}{\eta_{I_0}(M)}$ is $I_0$-torsion, so there is some $s$ such that $I_0^s f\subseteq \text{Im}{\eta_{I_0}(M)}$. Since $f=y^t c$, we have $(yI_0)^{\max(s,t)}c \subseteq \text{Im}{\gamma(M)}$. The element of $C$ represented by $c$ therefore belongs to $\Gamma_{yI_0}(C) = \Gamma_{yR\cap I_0}(C)$.
    \end{itemize}
    Lemma \ref{fac}(ii) therefore gives a map $\varphi(M): D_{I_0}(M)_y\to D_I(M)$, specifically $\varphi(M)=D_I(\gamma(M))^{-1}\circ \eta_{I}(D_{I_0}(M)_y)$. Both of the composite maps come from natural transformations $D_I(\gamma(-))^{-1}$ and $\eta_I(D_{I_0}(-)_y)$, so the result is a natural transformation
    $\varphi(-):D_{I_0}(-)_y\to D_I(-)$.

    It remains to show that that $\varphi(M)$ is an isomorphism for each $M$, which is equivalent, by Lemma \ref{fac}(iii), to showing that $\Gamma_I(D_{I_0}(M)_y)=H^1_I(D_{I_0}(M)_y)=0$. This can be done using the Mayer-Vietoris sequence associated with the intersection $yR\cap I_0$. Each module $H^i_{yR+I_0}(D_{I_0}(M)_y)$ vanishes because $y\in yR+I_0$ acts as a unit on
$D_{I_0}(M)_y$, and likewise for the modules $\Gamma_{yR}(D_{I_0}(M)_y)$ and $H^1_{yR}(D_{I_0}(M)_y)$. Note that $\text{depth}_{I_0}(D_{I_0}(M))\geq 2$, and localization can only make depth go up, so, $\Gamma_{I_0}(D_{I_0}(M)_y)=H^1_{I_0}(D_{I_0}(M)_y)=0$.
\begin{center}{
\begin{tikzpicture}[descr/.style={fill=white,inner sep=1.5pt}]
        \matrix (m) [
            matrix of math nodes,
            row sep=1em,
            column sep=2.5em,
            text height=1.5ex, text depth=0.25ex
        ]
        { 0 & 0 & 0 & \Gamma_{yR\cap I_0}(D_{I_0}(M)_y)\\
           & 0 & 0 & H^1_{yR\cap I_0}(D_{I_0}(M)_y)\\
           & 0 & 0\oplus H^2_{I_0}(D_{I_0}(M)_y) & H^2_{yR\cap I_0}(D_{I_0}(M)_y)\\
        };

        \path[overlay,->, font=\scriptsize,>=latex]
        (m-1-1) edge (m-1-2)
        (m-1-2) edge (m-1-3)
        (m-1-3) edge (m-1-4)
        (m-1-4) edge[out=350,in=170] (m-2-2)
        (m-2-2) edge (m-2-3)
        (m-2-3) edge (m-2-4)
        (m-2-4) edge[out=350,in=170] (m-3-2)
        (m-3-2) edge (m-3-3)
        (m-3-3) edge (m-3-4)
;
\end{tikzpicture}}
\end{center}
We can now see that $\Gamma_I(D_{I_0}(M)_y)=H^1_I(D_{I_0}(M)_y)=0$, as desired.
\end{proof}

\begin{corol}\label{depth1coh}
Let $R$ be a Noetherian ring, $y$ an element of $R$, and $I_0\subseteq R$ an ideal. Let $I=yR\cap I_0$. Then for all $i\geq 2$, there is a natural isomorphism of functors $H^i_I(-)\simeq H^i_{I_0}(-)_y$.
\end{corol}
\begin{proof}
  It is equivalent to show that $\mathscr{R}^{i-1}D_I(-)\simeq \left(\mathscr{R}^{i-1}D_{I_0}(-)\right)_y$. We
  can calculate $\mathscr{R}^{i-1}D_I(M)$ as $H^{i-1}(D_I(E^\bullet))$ where $M\to E^\bullet$ is an injective resolution, but by Proposition \ref{loc}, $D_I(-)\simeq D_{I_0}(-)_y$ where $(-)_y$ commutes with the formation of cohomology. Thus, $$H^{i-1}(D_I(E^\bullet))\simeq H^{i-1}(D_{I_0}(E^\bullet))_y= \left(\mathscr{R}^{i-1}D_{I_0}(M)\right)_y.$$
\end{proof}

\begin{theor}\label{H2}
  Let $R$ be a locally almost factorial Noetherian normal ring, and $I$, $J$ be ideals of $R$.
  The set $\text{Ass}\, H^2_I(J)$ is finite.
\end{theor}
\begin{proof}
  $R$ is a product of normal domains $R_1\times\cdots\times R_k$, and $J$ is a product of ideals
  $J_1\times\cdots\times J_k$ with $J_i\subseteq R_i$. It is enough to show that $\text{Ass}\, H^2_{IR_i}(J_i)$ is finite for all $i$, so assume that $R$ is a domain. By Lemma \ref{depth1}, we need only deal with the case in which $\text{depth}_I(R)=1$. It is enough to show that $\text{Ass}\, H^2_{IR_f}(J_f)$ is finite
  for each chart $\text{Spec}(R_f)$ in a finite cover of $\text{Spec}(R)$. By Corollary \ref{localprinc}, working with one chart at a time, and replacing $R$ by $R_f$ and $I$ by an ideal with the same radical,
  we may assume that $I$ has the form $I=yR\cap I_0$ where $\text{depth}_{I_0}(R)\geq 2$.
  By Corollary \ref{depth1coh}, this decomposition gives $H^2_I(J)\simeq H^2_{I_0}(J)_y$. It is therefore enough to show that $\text{Ass}\, H^2_{I_0}(J)$ is finite. But $\text{depth}_{I_0}(R)\geq 2$, so this follows at once from Lemma \ref{depth1}.
\end{proof}

\section{An example of an infinite $\text{Ass}\, H^3_I(J)$}
Let $K$ be a field, $A = K[u,v,w,x,y,z]$, and $f=wv^2 x^2 - (w+z)vxuy + zu^2y^2$. Katzman \cite{katzinf} showed that $\text{Ass}\, H^2_{(x,y)}(A/f)$ is infinite. The hypersurface ring $S=A/f$ can be presented as a complete intersection ring of higher codimension. For example, if $R = A[t]$ and $J=(t,f)R$, then clearly $S \simeq R/J$. Also, if $I = (t,f,x,y)R$, then $IS=(x,y)S$ and thus $H^i_{I}(S)\simeq H^i_{(x,y)}(S)$ for all $i$. In this case, $\text{depth}_I(R) = 3$ (the sequence $t,f,x\in I$ is $R$-regular), so the long exact sequence from applying $\Gamma_I(-)$ to $0 \to J \to R \to S\to 0$ begins with
\begin{center}{
\begin{tikzpicture}[descr/.style={fill=white,inner sep=1.5pt}]
        \matrix (m) [
            matrix of math nodes,
            row sep=1em,
            column sep=2.5em,
            text height=1.5ex, text depth=0.25ex
        ]
        { 0 & 0 & 0 & 0 \\
            & H^1_I(J) & 0 & H^1_{(u,v)}(S) \\
          & H^2_I(J) & 0 & H^2_{(u,v)}(S) \\
          & H^{3}_I(J) & H^3_I(R) & 0 \\
          & H^{4}_I(J) & H^{4}_I(R) & 0 \\
          & \vdots & \vdots & \vdots \\
        };

        \path[overlay,->, font=\scriptsize,>=latex]
        (m-1-1) edge (m-1-2)
        (m-1-2) edge (m-1-3)
        (m-1-3) edge (m-1-4)
        (m-1-4) edge[out=345,in=165] (m-2-2)
        (m-2-2) edge (m-2-3)
        (m-2-3) edge (m-2-4)
        (m-2-4) edge[out=345,in=165] (m-3-2)
        (m-3-2) edge (m-3-3)
        (m-3-3) edge (m-3-4)
        (m-3-4) edge[out=345,in=165] (m-4-2)
        (m-4-2) edge (m-4-3)
        (m-4-3) edge (m-4-4)
        (m-4-4) edge[out=345,in=165] (m-5-2)
        (m-5-2) edge (m-5-3)
        (m-5-3) edge (m-5-4)
        (m-5-4) edge[out=345,in=165] (m-6-2)
;
\end{tikzpicture}}
\end{center}
From this, we see that $H^2_{(x,y)}(S)$ is isomorphic to a submodule of $H^3_I(J)$. $\text{Ass}\, H^2_{(x,y)}(S)$ is infinite, so $\text{Ass}\, H^3_I(J)$ must be infinite as well. Note that this example occurs when $J$ is generated by a regular sequence of length $2$.


This approach can be used to produce examples of parameter ideals $J$ with a similar relationship to $R$ and $S$. If $S = A/J_0$ is a complete intersection ring of dimension $n$, with $A$ an LC-finite regular ring and $J_0$ an ideal generated by a regular sequence, we can let $R = A[z_1,\cdots,z_{N}]$ for $N\gg 0$, and let $J=(z_1,\cdots,z_N)R + J_0R$, and will have $S\simeq R/J$. For any ideal $I_0\subseteq A$, set $I = J + I_0 R$ and $H^i_{I_0}(S) \simeq H^i_{I}(S)$, with $d=\text{depth}_I(R) > n+1$ if $N$ is large enough. Using the long exact sequence from applying $\Gamma_I(-)$ to $0\to J \to R \to S\to 0$, it follows at once that $H^i_I(J) \simeq H^{i-1}_{I_0}(S)$ for all $1\leq i \leq n+1$, that $H^i_I(J) = 0$ for $n+2\leq i \leq d-1$, and that $H^i_I(J) \simeq H^i_I(R)$ for all $i\geq d$. In sufficiently large cohomological degrees, $\text{Ass}\, H^i_I(J)$ is finite, and in degrees $1\leq 1\leq n+1$, the local cohomology of $J$ has exactly the same finiteness properties as the local cohomology of $S$.

\section{Finiteness of $\text{Ass}\, H^i_I(J)$ vs finiteness of $\text{Ass}\, H^{i-1}_I(R/J)$}
The class of examples presented in Section 3 may be somewhat unsatisfying, since the infinite collection of embedded primes we found in $H^i_I(J)$ was directly inherited from $H^{i-1}_I(R/J)$. It is an interesting and unresolved question whether there exist modules $H^i_I(J)$ that exhibit this type of behavior \textit{even when} $H^{i-1}_I(R/J)$ is well-controlled. To be precise, the following question is, to the best of our knowledge, still open:

\vspace{0.5em}
\noindent\textbf{Question 3.} Let $R$ be an LC-finite regular ring, $J\subseteq R$ be an ideal generated by a regular sequence, and $I$ be an ideal of $R$ containing $J$. Does the finiteness of $\text{Ass}\, H^{i-1}_I(R/J)$ imply the finiteness of $\text{Ass}\, H^i_I(J)$?
\vspace{0.5em}

If $J$ is generated by a regular sequence of length $1$, then $J\simeq R$ as an $R$-module, and $\text{Ass}\, H^i_I(J)$ is finite, so the claim is trivially true. We therefore restrict our attention to when $\text{depth}_J(R)\geq 2$. We think of $i$ as being fixed with $I$ varying. The case $i = 2$ is completely answered, since $\text{Ass}\, H^1_I(R/J)$ is finite (as is true of $H^1_I(M)$ for any finitely generated module $M$), and $\text{Ass}\, H^2_I(J)$ is finite by Theorem \ref{H2}. The goal of this section is to give a partial positive answer to this question when $i=3$ and when $i=4$. As $i$ gets larger, our results require increasingly restrictive hypotheses on $R/J$.

To begin, notice that we can very easily ignore ideals $I$ where the depth of $R$ on $I$ is too large.

\begin{lem}\label{largedep}
Let $R$ be a Noetherian ring, let $I$ and $J$ be ideals of $R$, and let $S=R/J$. Fix $i\geq 1$ and assume $\text{Ass}\, H^i_I(R)$ is finite. If $\text{depth}_I(R)>i-1$, then $\text{Ass}\, H^{i}_I(J)$ is finite if and only if $\text{Ass}\, H^{i-1}_I(R/J)$ is finite.
\end{lem}
\begin{proof}
There is a short exact sequence
$
0 \to H^{i-1}_I(R/J) \to H^{i}_I(J) \to N \to 0
$
where $N\subseteq H^i_I(R)$, so $\text{Ass}\, N$ is finite.
\end{proof}

We may therefore restrict our focus to the case where $\text{depth}_I(R)\leq i-1$. We will show that it actually suffices to only consider ideals $I$ such that $\text{depth}_I(R) =i-1$. This crucial simplification is inspired by a very similar strategy employed by Hellus, albeit for different applications \cite{hellus}.

\subsection{A generalized isomorphism of Hellus}
Before proceeding, we require a notion of parameters in a global ring, and the following lemma provides one suitable enough for use in our main proofs.
\begin{lem}
  Let $R$ be a Noetherian ring, let $I$ be a proper ideal of height $h\geq 0$, and let $J\subseteq I$
  be an ideal of height $j\geq 0$.
  \begin{enumerate}[(a)]
  \item If an ideal of the form $(x_1,\cdots,x_h)R$ has height $h$, then it has pure height $h$.
  \item Any sequence $x_1,\cdots,x_j\in J$
    generating an ideal of height $j$ (including the empty sequence if $j=0$) can be extended to a sequence $x_1,\cdots,x_h\in I$ generating an ideal of height $h$.
  \item There is a sequence $x_1,\cdots,x_h\in I$ such that $(x_1,\cdots,x_h)R$ has (necessarily pure)
    height $h$.
  \end{enumerate}
\end{lem}
\begin{proof}
  (a) Every minimal prime of a height $h$ ideal has height at least $h$, and every minimal prime
  of an $h$-generated ideal has height at most $h$.
    
  (b) If $j=h$, there is nothing to do, so assume $j<h$. By induction, it is enough to show that
  we can extend the sequence by one element. Since $j<h$, $I$ is not contained in any minimal
  prime of $(x_1,\cdots,x_j)R$ (all of which have height $j$), and so we may choose $x\in I$
  avoiding all such primes. A height $j$ prime containing $(x_1,\cdots,x_j)R$
  therefore cannot also contain $xR$. Thus, the minimal primes of $(x,x_1,\cdots,x_j)R$ have
  height at least $j+1$. By Krull's height theorem, they also have height at most $j+1$.
  
  (c) This follows at once from (b) by taking $J=(0)$.
\end{proof}
\noindent By convention, we take the height of the unit ideal to be $+\infty$. An intersection
of prime ideals of $R$ indexed by the empty set is taken to be $\bigcap_{i\in\varnothing} P_i=R$. If $R$ is a ring and $I$ is an ideal, let $\text{ara}_R(I)$ denote the arithmetic rank of $I$, i.e. the minimum number of generators among all ideals with the same radical as $I$.
\begin{lem}\label{chooseelt}
  Let $R$ be a Noetherian ring, let $I$ be a proper ideal of height $h$, and let $J\subseteq I$ be
  an ideal of height $j\leq h$ generated by $j$ elements. There is an element $y\in R$
  that satisfies the following properties.
  \begin{enumerate}[(i)]
  \item $\text{ara}_R(yR\cap I)\leq h$
  \item $\text{ara}_{R/J}(y(R/J)\cap(I/J))\leq h-j$
  \item $\text{ht}(yR+I)\geq h+1$
  \end{enumerate}
\end{lem}
\begin{proof}
Write $J=(x_1,\cdots,x_j)R$, and extend this sequence to
  $x_1,\cdots,x_h\in I$ generating an ideal of height $h$. $I$ is contained in at least one minimal prime of $(x_1,\cdots,x_h)R$. 

Let $P_1,\cdots,P_t$ be the minimal primes of $(x_1,\cdots,x_h)R$ containing $I$, and $Q_1,\cdots,Q_s$ be the minimal primes of $(x_1,\cdots,x_h)R$ that do not. We may have $s=0$. Since these primes are pairwise incomparable, there exist elements $y\in Q_1\cap\cdots\cap Q_s$ that avoid the union $P_1\cup\cdots\cup P_t$ (if $s=0$, we can take $y=1$). For any such $y$, it holds that $$yR \cap I\subseteq P_1\cap\cdots\cap P_t\cap Q_1\cap\cdots\cap Q_s=\sqrt{(x_1,\cdots,x_h)R}$$
and thus $yR\cap I\subseteq yR\cap \sqrt{(x_1,\cdots,x_h)R}$.
Since $(x_{1},\cdots,x_h) \subseteq I$, we see that
    $$
    yR \cap (x_{1},\cdots,x_h)R \subseteq yR \cap I \subseteq yR \cap \sqrt{(x_{1},\cdots,x_h)R}
    $$
    It follows at once that that $$\sqrt{yR\cap I}=\sqrt{yR\cap(x_{1},\cdots,x_h)R}=\sqrt{(yx_{1},\cdots,yx_h)R}$$ producing the desired bound on arithmetic rank: $\text{ara}_R(yR \cap I)\leq h$. Since $yR/J\cap I/J\subseteq \sqrt{(x_{j+1}\cdots,x_h)R/J}$, an identical argument to the above shows that
    $$\sqrt{yR/J\cap I/J} = \sqrt{(yx_{j+1},\cdots,yx_h)R/J}$$ so $\text{ara}_{R/J}(yR/J\cap I/J)\leq h-j$. We have established (i) and (ii). Concerning (iii), note that that all primes containing $yR+I$ also contain $(x_1,\cdots,x_h)R$, and thus, to show that $\text{ht}(yR+I)\geq h+1$, it is enough to show that none of the height $h$ primes containing $(x_1,\cdots,x_h)R$ appear in $V(yR+I)$.
    But this is clear, since $\{P\supseteq (x_1,\cdots,x_h)R\hspace{0.1cm}|\hspace{0.1cm} \text{ht}(P)=h\} =\{P_1,\cdots,P_t,Q_1,\cdots,Q_s\}$. None of the primes $P_i$ contain $yR$, and none of the
    primes $Q_j$ contain $I$.
\end{proof}
The following is a generalization of Hellus's isomorphism \cite[Theorem 3]{hellus}. We significantly relax the hypotheses (originally $R$ was assumed to be a Cohen-Macaulay local ring) and obtain an isomorphism of functors (originally the isomorphism was of modules $H^i_{I^\prime}(R)\to H^i_I(R)$).
\begin{theor}\label{mainj}
  Let $R$ be a Noetherian ring, let $I$ be an ideal of height $h$, and let $J\subseteq I$
  be an ideal of height $j\geq 0$ generated by $j$ elements. For any $k\geq 0$, there is an ideal $I_{k,J}\supseteq I$ such that
  \begin{itemize}
  \item $\text{ht}(I_{k,J}) \geq \text{ht}(I)+k$
  \item The natural transformation
    \begin{center}
      \begin{tikzcd}[column sep = 2.0em]
        H^i_{I_{k,J}}(-)\ar[r] & H^i_{I}(-)
      \end{tikzcd}
    \end{center}
    is an isomorphism on $R$-modules for all $i>h+k$, and an isomorphism on $R/J$-modules for
    all $i>h-j+k$. If $i=h+k$ (resp. $i=h-j+k$) this natural
    transformation is a surjection on $R$-modules (resp. $R/J$-modules).
  \end{itemize}
\end{theor}
\begin{proof}
  If $k=0$, choose $I_{0,J}=I$. Fix $k\geq 1$, and suppose that we've
  chosen the ideal $I_{k-1,J}$ by induction. We must choose $I_{k,J}$. For brevity, we will suppress
  $J$ from our notation, and write $I_{k-1}$ and $I_k$ for $I_{k-1,J}$ and $I_{k,J}$, respectively.

If $\text{ht}(I_{k-1}) > h + k-1$ we can simply pick $I_{k} = I_{k-1}$, so assume that $\text{ht}(I_{k-1}) = h + k-1$. By Lemma \ref{chooseelt} there is an element $y\in R$ such that $\text{ht}(yR+I_{k-1})\geq (h+k-1)+1$, with
  $$\text{ara}_R(yR\cap I_{k-1})\leq h+k-1\hspace{0.5em}\text{and}\hspace{0.5em} \text{ara}_{R/J}(y(R/J)\cap I_{k-1}/J)\leq (h+k-1)-j$$
 Consider the Mayer-Vietoris sequence on the intersection $yR\cap I_{k-1}$. We use $(-)$ in our notation to mean that the sequence is exact when $-$ is replaced by any $R$-module $M$, and that all maps in the sequence are given by natural transformations.
\begin{center}{
\begin{tikzpicture}[descr/.style={fill=white,inner sep=1.5pt}]
        \matrix (m) [
            matrix of math nodes,
            row sep=1em,
            column sep=2.5em,
            text height=1.5ex, text depth=0.25ex
        ]
        { &  & \cdots & H^{i-1}_{yR\cap {I_{k-1}}}(-) \\
            & H^{i}_{yR+{I_{k-1}}}(-) & H^{i}_{yR}(-)\oplus H^i_{{I_{k-1}}}(-) & H^{i}_{yR\cap {I_{k-1}}}(-) \\
        };

        \path[overlay,->, font=\scriptsize,>=latex]

        (m-1-3) edge (m-1-4)
        (m-1-4) edge[out=350,in=170] (m-2-2)
        (m-2-2) edge (m-2-3)
        (m-2-3) edge (m-2-4)
;
\end{tikzpicture}}
\end{center}
Let $i>h+k$. Since $i-1>\text{ara}_{R}(yR\cap I_{k-1})$, we get vanishing $H^{i-1}_{yR\cap I_{k-1}}(-)=H^i_{yR\cap I_{k-1}}(-)=0$. Since $i\geq h+k+1\geq 2$, we also have $H^i_{yR}(-)=0$, and therefore obtain a natural isomorphism $H^i_{yR+I_{k-1}}(-)\xrightarrow{\sim} H^i_{I_{k-1}}(-)$. Notice that if $i=h+k$, then we still have $H^i_{yR\cap I_{k-1}}(-)=0$, so
$$H^i_{yR+I_{k-1}}(-)\to H^i_{yR}(-)\oplus H^i_{I_{k-1}}(-)\to 0$$
is exact, implying that the component map $H^i_{yR+I_{k-1}}(-)\to H^i_{I_k}(-)$ is surjective. Working with $R/J$-modules, an identical argument using the fact that $$\text{ara}_{R/J}(y(R/J)\cap I_{k-1}/J)\leq (h+k-1)-j$$ shows that
\[H^i_{y(R/J)+I_{k-1}/J}(-)\xrightarrow{\sim} H^i_{I_{k-1}/J}(-)\] when $i>h+k-j$ and
\[H^i_{y(R/J)+I_{k-1}/J}(-)\twoheadrightarrow H^i_{I_{k-1}/J}(-)\] when $i=h+k-j$. Finally,
$\text{ht}(yR + I_{k-1})\geq h+k$, so we may in fact choose $I_k = yR + I_{k-1}$, which completes the induction.
\end{proof}
\begin{corol}\label{genhel}
Let $R$ be a Noetherian ring and let $I\subseteq R$ be any ideal. Fix $i\geq 0$. There is an ideal $I^\prime\supseteq I$ (resp. $I^{\prime\prime}\supseteq I$) such that
\begin{itemize}
\item $\text{ht}(I^\prime)\geq i-1$ (resp. $\text{ht}(I^{\prime\prime})\geq i$)
\item $H^i_{I^\prime}(-) \xrightarrow{\sim} H^i_{I}(-)$ (resp. $H^i_{I^{\prime\prime}}(-) \twoheadrightarrow H^i_{I}(-)$)
\end{itemize}
\end{corol}
\begin{proof}
  Let $h=\text{ht}(I)$. If $h\geq i-1$ (resp. $h\geq i$) simply choose $I^\prime=I$ (resp. $I^{\prime\prime}=I$). Otherwise, $h<i-1$ (resp. $h < i$). Apply Theorem \ref{mainj} in the case $k=i-1-h$ (resp. $k^\prime=i-h$) to obtain an ideal $I^\prime\supseteq I$ (resp. $I^{\prime\prime}\supseteq I$) satisfying $\text{ht}(I^\prime)\geq h+k=i-1$ (resp. $\text{ht}(I^{\prime\prime})\geq h+k^\prime=i$) and $H^i_{I^\prime}(-)\xrightarrow{\sim}H^i_I(-)$, since $i>h+k$ (resp. $H^i_{I^{\prime\prime}}(-)\twoheadrightarrow H^i_I(-)$, since $i=h+k^\prime$). 
\end{proof}

An immediate application of this theorem is to generalize a corollary of Hellus \cite[Corollary 2]{hellus}. This generalization provides a new proof of a result of Marley \cite[Proposition 2.3]{marley}, namely, for any Noetherian ring $R$, any ideal $I\subseteq R$, and any $R$-module $M$, $\{P\in\text{Supp}\, H^i_I(M)\,|\,\text{ht}(P)=i\}$ is a finite set. Since our result comes from a surjection of functors, we will describe it in terms of the ``support'' of $H^i_I(-)$.

By the \textit{support} of a functor $F:\text{Mod}_R\to\text{Mod}_R$, we mean the set of primes $P\in\text{Spec}(R)$ such that $F(-)_P$ is
not the zero functor. That is to say,
$$
\text{Supp}(F) := \{P\in\text{Spec}(R)\,|\, \exists M\in\text{Mod}_R \text{ such that } F(M)_P \neq 0\}
$$
For example, if $I\subseteq R$ is an ideal and $i\geq 0$, then $\text{Supp}\, H^i_I(-)\subseteq V(I)$. If
$i>\text{ht}(I)$, this inclusion is not sharp. The following Corollary shows us how to find a
strictly smaller closed set containing $\text{Supp}\, H^i_I(-)$.
\begin{corol}\label{heightcontrol} (cf. Marley \cite[Proposition 2.3]{marley})
  Let $R$ be a Noetherian ring and $I$ be an ideal. Then for all $i\geq 0$, there is an ideal $I^{\prime\prime}\supseteq I$ with $\text{ht}(I^{\prime\prime})\geq i$ such that $\text{Supp}\, H^i_I(-)\subseteq V(I^{\prime\prime})$. In particular, for any $R$-module $M$, the set
  $$\{P\in\text{Supp}\, H^i_I(M)\,\,|\,\, \text{ht}(P)= i\}$$
  is a subset of $\text{Min}_R(R/I^{\prime\prime})$, and is therefore finite. If $R$ is semilocal and $i\geq \dim(R)-1$, then $\text{Supp}\, H^i_I(M)$ is a finite set.
\end{corol}
\begin{proof}
  Fix $i\geq 0$ and write $h=\text{ht}(I)$. If $i<h$, then because $\text{Supp}\, H^i_I(-)\subseteq V(I)$, we already have $\text{ht}(P) \geq h>i$ for all $P\in \text{Supp}\, H^i_I(-)$ and there is nothing to prove.
  So assume that $i\geq h$. By Corollary \ref{genhel}, there is an ideal
  $I^{\prime\prime}\supseteq I$ such that $\text{ht}(I^{\prime\prime})\geq i$ and $H^i_{I^{\prime\prime}}(-)\twoheadrightarrow H^i_I(-)$. In particular, for any $R$-module $M$, $H^i_I(M)$ is $I^{\prime\prime}$-torsion, and thus $\text{Supp}\, H^i_I(M)\subseteq V(I^{\prime\prime})$. All primes in $V(I^{\prime\prime})$ have height at least $i$. Any primes of height exactly $i$ must be
  among the minimal primes of $I^{\prime\prime}$, of which there are only finitely many.
\end{proof}

\subsection{Comparing the local cohomology of $J$ and $R/J$}
Of primary interest within the scope of this paper is the following application of Theorem \ref{mainj} to the setting of Question 3.
\begin{corol}\label{reduction}
  Let $R$ be a Cohen-Macaulay ring, and let $J\subseteq R$ be an ideal generated by a regular sequence of length $j\geq 1$. Fix a nonnegative integer $i$, and let $I$ be an ideal containing $J$ such that $\text{depth}_{I}(R)\leq i-1$. Then there is an ideal $I^\prime\supseteq I$ such that $\text{depth}_{I^\prime}(R)\geq i-1$ and such that $H^i_{I^\prime}(J)\simeq H^i_{I}(J)$ and $H^{i-1}_{I^\prime}(R/J)\simeq H^{i-1}_I(R/J)$.
\end{corol}
\begin{proof}
  Write $h=\text{ht}(I)$ and $j=\text{ht}(J)$.
  Apply Theorem \ref{mainj} with $k=i-1-h$ to obtain $I^\prime\supseteq I$ such that
  $\text{depth}_{I^\prime}(R) = \text{ht}(I^\prime) \geq i-1$ and such that the natural transformation
  $H^\ell_{I^\prime}(-)\to H^\ell_{I}(-)$ is an isomorphism on $R$-modules whenever $\ell>i-1$ and on $R/J$-modules whenever $\ell>i-1-j$. In particular, we see that $H^i_{I^\prime}(-)\to H^i_{I}(-)$ is an isomorphism on $R$-modules $H^{i-1}_{I^\prime}(-)\to H^{i-1}_I(-)$ is an isomorphism on $R/J$-modules.
\end{proof}

Assume that $R$ is Cohen-Macaulay, $J$ is generated by a regular sequence of length $j\geq 2$, and $I\supseteq J$ is any ideal. Fix $i\geq 0$ and let $a=\text{depth}_I(R/J)=\text{depth}_I(R)-j$. If $a+j\leq i-1$, then by Corollary \ref{reduction}, we can replace $I$ with a possibly larger ideal in order to assume that $a+j\geq i-1$, without affecting $H^i_I(J)$ and $H^{i-1}_I(R/J)$. Lemma \ref{largedep} gives a positive answer to Question 3 if $a+j>i-1$, so we may assume that $a+j=i-1$. Note that this allows us to ignore all values of $i$ and $j$ for which $j>i-1$.  Here is a table illustrating the relevant values of $a$ to consider for various small values of $i$ and $j$.
\begin{center}
  \begin{tabular}{l|| c| c| c| c| c }    
         & $i=3$         & $i=4$          & $i=5$         & $i=6$         &  $i=7$\\\hline\hline
    $j=2$& $a=0$         &  $a=1$         & $a=2$         & $a=3$         &  $a=4$\\\hline
    $j=3$& $\varnothing$ &  $a=0$         & $a=1$         & $a=2$         &  $a=3$\\\hline
    $j=4$& $\varnothing$ &  $\varnothing$ & $a=0$         & $a=1$         &  $a=2$\\\hline
    $j=5$& $\varnothing$ &  $\varnothing$ & $\varnothing$ & $a=0$         &  $a=1$\\\hline
    $j=6$& $\varnothing$ &  $\varnothing$ & $\varnothing$ & $\varnothing$ &  $a=0$\\\hline
    $j=7$& $\varnothing$ &  $\varnothing$ & $\varnothing$ & $\varnothing$ &  $\varnothing$\\
\end{tabular}
\end{center}

\noindent We will attack the cases $a=0$ and $a=1$ directly, and the next lemma is our main tool in doing so.

\begin{lem}\label{prinq3}
  Fix $a\geq 0$. Let $R$ be a Noetherian ring, let $J\subseteq R$ be an ideal generated by a regular sequence of length $j\geq 2-a$, let $I$ be an ideal containing $J$, and let $S=R/J$. Suppose that $IS$ can be decomposed (up to radicals) as $yS\cap I_0$ with $\text{depth}_{I_0}(S)> a$. Suppose further that
  $\text{Ass}\, H^{j+a+1}_I(R)$ is finite. Then $\text{Ass}\, H^{j+a+1}_I(J)$ is finite if and only if $\text{Ass}\, H^{j+a}_I(S)$ is finite.
\end{lem}
\begin{proof}
  By Corollary \ref{depth1coh}, there is a natural isomorphism
  $H^i_{I_0}(S)_y\simeq H^i_I(S)$ for all $i\geq 2$, so
  that in particular, $H^{j+a}_I(S)$ is an $S_y$-module. The natural map $\psi:H^{j+a}_I(R)\to H^{j+a}_I(S)$ therefore factors through $H^{j+a}_I(R)\to S_y\otimes_R H^{j+a}_I(R)$ to give an $S_y$-linear map
  $S_y\otimes_R H^{j+a}_I(R)\to H^{j+a}_I(S)$. 
  \begin{center}
    \begin{tikzcd}
      H^{j+a}_I(R)\ar[r,"\psi"]\ar[dr] & H^{j+a}_I(S)\\
      & S_y\otimes_R H^{j+a}_I(R)\ar[u]
    \end{tikzcd}
  \end{center}
  We claim that $\psi=0$, and for this it suffices to show that $S_y\otimes_R H^{j+a}_I(R)=0$.

  Consider the decomposition of $I$ up to radicals as $yS\cap I_0$ in $S$. We can replace $y$ by some
  lift mod $J$ to assume that $y\in R$, and since $I_0$ is expanded
  from $R$, we can write $I_0=I_0^\prime S$ for some ideal $I_0^\prime$ of $R$
  containing $J$. We therefore have $I=(yR+J)\cap I_0^\prime$ in $R$ (after possibly replacing $I$ by an ideal with the same radical). Note that $\text{depth}_{I_0^\prime}(R)> j+a$. We can write $S_y\otimes_R H^{j+a}_I(R) = S_y\otimes_{R_y} H^{j+a}_{IR_y}(R_y)$ where $IR_y = (y + J)R_y \cap I_0^\prime R_y = I_0^\prime R_y$, and thus $H^{j+a}_{I R_y}(R_y) = H^{j+a}_{I_0^\prime}(R)_y$. Since $\text{depth}_{I_0^\prime}(R)>j+a$, we have $H^{j+a}_{I_0^\prime}(R)=0$ and consequently, $\psi=0$.

  We therefore have an exact sequence
  $$
  0\to H^{j+a}_I(S)\to H^{j+a+1}_I(J)\to H^{j+a+1}_I(R).
  $$
  Since $\text{Ass}\,{H^{j+a+1}_I(R)}$ is finite, the claim follows at once.
\end{proof}

We can now prove the main result of this section.

\begin{theor}\label{smalli}
  Let $R$ be an LC-finite regular ring, let $J\subseteq R$ be an ideal generated generated by a regular sequence of length $j\geq 2$, and let $S=R/J$. For any ideal $I$ of $R$ containing $J$,
  \begin{enumerate}
  \item[(i)] $\text{Ass}\, H^i_I(J)$ and $\text{Ass}\, H^{i-1}_I(S)$ are always finite for $i\leq 2$.
  \item[(ii)] If the irreducible components of $\text{Spec}(S)$ are disjoint, then $\text{Ass}\, H^{3}_I(J)$ is finite if and only if $\text{Ass}\,{H^2_I(S)}$ is finite.
  \item[(iii)] If $S$ is normal and locally almost factorial, then $\text{Ass}\,{H^4_I(J)}$ is finite if and only if $\text{Ass}\, H^3_I(S)$ is finite.
  \end{enumerate}
\end{theor}
\begin{proof}
  For (i), it is well known that for any finitely generated $R$-module $M$, $\text{Ass}\, H^i_I(M)$ is finite whenever $i\leq 1$. The finiteness of $\text{Ass}\, H^2_I(J)$ is the subject of Theorem \ref{H2}.

  For (ii), we may use Corollary \ref{reduction} to replace $I$ with a possibly larger ideal
  in order to assume that $\text{depth}_I(R) \geq 2$. By Lemma \ref{largedep}, (ii) is immediate if $\text{depth}_I(R)> 2$, so assume $\text{depth}_I(R)= 2$. Since $J\subseteq I$ and $\text{depth}_J(R)\geq 2$, it follows
  that $\text{depth}_J(R)=2$ and $\text{depth}_I(S)=\text{ht}(IS)=0$. Let $e_1,\cdots,e_t\in S$ be a complete
  set of orthogonal idempotents. The minimal primes of $S$ are
  $\sqrt{(1-e_1)S},\cdots,\sqrt{(1-e_t)S}$, and thus, every pure height $0$ ideal of $S$
  has arithmetic rank at most $1$. Up to radicals, we can therefore write $IS$ as $yS\cap I_0$
  where $\text{ht}(I_0)=\text{depth}_{I_0}(S)>0$. Since $\text{depth}_J(R)\geq 2-0$, the claim follows from Lemma \ref{prinq3} in the case where $a=0$.

  For (iii), again using Corollary \ref{reduction} and Lemma \ref{largedep}, we may assume
  that $\text{depth}_I(R)=\text{depth}_I(S)+\text{depth}_J(R)=3$. Since $\text{depth}_J(R)\geq 2$, this means $\text{depth}_I(S)\leq 1$. If $\text{depth}_J(R)=3$, giving $\text{depth}_I(S)=0$, then we may argue as in the proof of (ii) (note that $S$ is a product of domains). If $\text{depth}_J(R)=2$, giving $\text{depth}_I(S)=1$, then by Corollary \ref{localprinc} there is a finite cover of $\text{Spec}(S)$ by charts
  $\text{Spec}(S_{f_1}),\cdots,\text{Spec}(S_{f_t})$ such that for each $i$, we can write (up to radicals) $IS_{f_i} = y_iS_{i}\cap I_{0,i}$ with $\text{depth}_{I_{0,i}}(S_{f_i})>1$. Replace $f_1,\cdots,f_t$ with lifts
  from $R/J$ to $R$ in order to assume $f_1,\cdots,f_t\in R$.
  Lemma \ref{prinq3} in the case $a=1$ shows that for each $i$,
  $\text{Ass}\, H^4_I(J)_{f_i}$ is finite if and only if $\text{Ass}\, H^{3}_I(S)_{f_i}$ is finite.
  The charts $\text{Spec}(R_{f_1}),\cdots,\text{Spec}(R_{f_t})$ do not necessarily cover $\text{Spec}(R)$, but they do cover the subset $V(J)$. Since $I\supseteq J$, $\text{Supp}\, H^\ell_I(-)\subseteq V(I)\subseteq V(J)$ for all $\ell$,
  so showing that $\text{Ass}\, H^4_I(J)$ is finite is equivalent to showing that $\text{Ass}\, H^4_I(J)_{f_i}$ is finite for each $i$. The result we proved on each chart therefore implies
  $\text{Ass}\, H^4_I(J)$ is finite if and only if $\text{Ass}\, H^3_I(S)$ is finite.
\end{proof}

Under the hypotheses of (iii), we can give the following partial answer to Question 3 for local rings of sufficiently small dimension.
\begin{corol}
Let $(R,\mathfrak{m},K)$ be an LC-finite regular local ring of dimension at most $7$, let $J$ be an ideal generated by a regular sequence of length at least $2$ such that $S=R/J$ is normal and almost factorial. Let $I$ be any ideal of $R$ containing $J$. Then for all $i\geq 1$, $\text{Ass}\, H^i_I(J)$ is finite if and only if $\text{Ass}\, H^{i-1}_I(S)$ is finite.
\end{corol}
\begin{proof}
  The case $i\leq 4$ is the subject of Theorem \ref{smalli}. We must have $\dim(S)\leq 5$ since $\text{depth}_J(R)\geq 2$, so by Corollary \ref{heightcontrol}, $\text{Supp}\, H^{i-1}_I(S)$ (and hence $\text{Ass}\, H^{i-1}_I(S)$) is a finite set if $i-1\geq 4$. Likewise, for any homomorphic image $H^{i-1}_I(S)\twoheadrightarrow N$, the set $\text{Supp}\,(N)$ is finite. There is an exact sequence $0\to N \to H^i_I(J)\to M\to 0$
  where $N$ is a homomorphic image of $H^{i-1}_I(S)$ and $M$ is a submodule of $H^i_I(R)$. If $i\geq 5$, both $\text{Ass}\,(N)$ (a subset of $\text{Supp}\,(N)$) and $\text{Ass}\,(M)$ (a subset of $\text{Ass}\, H^i_I(R)$) are finite, so $\text{Ass}\, H^i_I(J)$ is finite as well.
\end{proof}


\section{Regular parameter ideals in characteristic $p>0$}
In this section, we will show that if $R$ is a regular ring of prime characteristic $p>0$ and $J$ is an ideal such that $R/J$ is regular, then for any ideal $I\subseteq R$ and any $i\geq 0$, the set $\text{Ass}\, H^i_I(J)$ is finite. This result is essentially a corollary of a stronger result. Specifically, we will show that if $R\to S$ is a homomorphism between two regular rings of prime characteristic $p>0$, and $I\subseteq R$ is an ideal, then for any $i\geq 0$, the natural map
$$
S\otimes_R H^i_I(R)\to H^i_I(S)
$$
is a morphism of $F$-finite $F_S$-modules in the sense of Lyubeznik \cite{lyufmod}. The proof of this statement relies on understanding a certain family of natural transformations that compare the cohomology of a complex before and after applying some base change functor.

\subsection{Natural transformations associated with cohomology and base change}
Let $\text{Mod}_{R}$ denote the category of modules over a ring $R$, and let $\text{Kom}_{R}$ denote the category
of cohomologically indexed complexes of $R$-modules. Let $H^i:\text{Kom}_{R}\to\text{Mod}_{R}$ denote the functor
that takes a complex $C^\bullet$ to its $i$th cohomology module $H^i(C^\bullet)$.

\begin{defn}
  If $f:R\to S$ is a ring homomorphism and $C^\bullet$ is an $R$-complex, then the natural ($R$-linear) map $C^\bullet \to S\otimes_R C^\bullet$ induces $H^i(C^\bullet)\to H^i(S\otimes_R C^\bullet)$, which factors uniquely through the natural map $H^i(C^\bullet)\to S\otimes_R H^i(C^\bullet)$ to an $S$-linear map $S\otimes_R H^i(C^\bullet)\to H^i(S\otimes_R C^\bullet)$. Call this map $h^i_{f}(C^\bullet)$, and let
  $h^i_{f}$ denote the corresponding natural transformation
  $$h^i_{f}(-):S\otimes_R H^i(-)\longrightarrow H^i(S\otimes_R -)$$
  of functors $\text{Kom}_{R}\to\text{Mod}_{S}$.
\end{defn}
If the homomorphism $f:R\to S$ is understood from context, we will write $h^i_{S/R}(-)$ instead of $h^i_{f}(-)$. The latter, more precise notation is reserved for ambiguous cases, such as when $R=S$ has prime characteristic $p$ and $f$ is the Frobenius homomorphism.

Note that $S$ is flat over $R$ if and only if $h^i_{S/R}(-)$ is an
  isomorphism of functors. Our goal in this section is to briefly review some straightforward but crucially important compatibility properties of these natural transformations.

\begin{prop}\label{comp} (Compatibility with composition)
  Let $R\to S\to T$ be ring homomorphisms. The following diagram of functors $\text{Kom}_{R}\to\text{Mod}_{T}$
  commutes
  \begin{center}
    \begin{tikzcd}[column sep = small]
      T\otimes_R H^i(-)\ar[rrrrrdd,"h^i_{T/R}(-)"'] \ar[r,no head, shift right = 0.1em]\ar[r,no head,shift left=0.1em]& T\otimes_S (S\otimes_R H^i(-))\ar[rrrr,"\text{id}_T\otimes_S h^i_{S/R}(-)"]
      & & & & T\otimes_S H^i(S\otimes_R -)\ar[d,"h^i_{T/S}(S\otimes_R -)"]\\
     & & & & & H^i(T\otimes_S(S\otimes_R -))\ar[d,no head, shift right = 0.1em]\ar[d,no head,shift left=0.1em]\\
     & & & & & H^i(T\otimes_R -)
    \end{tikzcd}
  \end{center}
  where the equalities shown above come from identifying the functor $T\otimes_S(S\otimes_R-)$
  with $T\otimes_R -$.
\end{prop}

\begin{prop}\label{simflat} (Simultaneous flat base change)
  Suppose there is a commutative square of ring homomorphisms
  \begin{center}
    \begin{tikzcd}
    R\ar[d]\ar[r] & S\ar[d]\\
    R^\prime\ar[r] & S^\prime
    \end{tikzcd}
  \end{center}
  such that $R^\prime$ is flat over $R$, and $S^\prime$ is flat over $S$. Then there is a commutative
  square of functors $\text{Kom}_{R}\to \text{Mod}_{S^\prime}$
  \begin{center}
    \begin{tikzcd}[column sep = huge]
      S^\prime\otimes_S (S\otimes_R H^i(-))\ar[d,"\text{\rotatebox{90}{$\sim$}}"'] \ar[r,"\text{id}_{S^\prime}\otimes_S h^i_{S/R}(-)"]
      & S^\prime\otimes_S H^i(S\otimes_R-)\ar[d,"\text{\rotatebox{90}{$\sim$}}"]\\
      
      S^\prime \otimes_{R^\prime} H^i(R^\prime\otimes_R -) \ar[r,"h^i_{S^\prime/R^\prime}(R^\prime\otimes_R -)"']
      & H^i(S^\prime\otimes_{R^\prime} (R^\prime\otimes_R-))
    \end{tikzcd}
  \end{center}
  where the vertical arrows are natural isomorphisms.
\end{prop}
\begin{proof} Apply Proposition \ref{comp} to $R\to S\to S^\prime$ in order to see that the
  upper right corner of the below diagram commutes, and then to $R\to R^\prime \to S^\prime$
  to see that the lower left corner commutes as well:
  \begin{center}
    \begin{tikzcd}[column sep = huge, row sep = huge]
      S^\prime\otimes_S (S\otimes_R H^i(-))\ar[r,"\text{id}_{S^\prime}\otimes_S h^i_{S/R}(-)"]\ar[d,no head, shift right = 0.1em]\ar[d,no head,shift left=0.1em]
      & S^\prime\otimes_S H^i(S\otimes_R-)\ar[d,"h^i_{S^\prime/S}(S\otimes_R-)"] \\
      
      S^\prime \otimes_R H^i(-)\ar[d,no head, shift right = 0.1em]\ar[d,no head,shift left=0.1em]
      \ar[rd,"h^i_{S^\prime/R}(-)"]
       & H^i(S^\prime\otimes_S (S\otimes_R -))\ar[d,no head, shift right = 0.1em]\ar[d,no head,shift left=0.1em] \\
         
      S^\prime \otimes_{R^\prime}(R^\prime\otimes_R H^i(-))
      \ar[d,"\text{id}_{S^\prime}\otimes_{R^\prime} h^i_{R^\prime/R}(-)"']
       & H^i(S^\prime\otimes_R -)\ar[d,no head, shift right = 0.1em]\ar[d,no head,shift left=0.1em]\\
       
       S^\prime \otimes_{R^\prime} H^i(R^\prime\otimes_R -)\ar[r,"h^i_{S^\prime/R^\prime}(R^\prime\otimes_R -)"']
      & H^i(S^\prime\otimes_{R^\prime} (R^\prime\otimes_R-))
    \end{tikzcd}
  \end{center}
  We are finished upon observing that the flatness of $R^\prime$ over $R$ (resp. of $S^\prime$ over $S$) implies that $h^i_{R^\prime/R}(-)$ (resp. $h^i_{S^\prime/S}(-)$) is a natural isomorphism.
\end{proof}

Our main interest is in applying Proposition \ref{simflat} in the case where $R$ and $S$ are regular rings of primes characteristic $p>0$, and $R\to R^\prime$, $S\to S^\prime$ are the ($e$-fold iterates of the) Frobenius homomorphisms of $R$ and $S$. Some notation: if $A$ is a ring of prime characteristic $p>0$, let $F_A:A\to A$ denote the Frobenius homomorphism of $A$, and $\mathcal{F}_A(-)$ (either $\text{Mod}_A\to \text{Mod}_A$ or $\text{Kom}_A\to \text{Kom}_A$, depending on context) denote the base change functor associated with $F_A$. Let $\mathcal{F}^e_A(-)$ denote the $e$-fold iterate of $\mathcal{F}_A(-)$.
\begin{corol}\label{frob} 
  Let $R\to S$ be a homomorphism between two regular rings of prime characteristic $p>0$. For
  each $e\geq 1$, the following diagram commutes,
  \begin{center}
    \begin{tikzcd}[column sep = huge]
      \mathcal{F}^e_S(S\otimes_R H^i(-))\ar[d,"\text{\rotatebox{90}{$\sim$}}"'] \ar[r,"\mathcal{F}^e_S(h^i_{S/R}(-))"]
      & \mathcal{F}^e_S(H^i(S\otimes_R -))\ar[d,"\text{\rotatebox{90}{$\sim$}}"]\\
      
      S \otimes_{R} H^i(\mathcal{F}^e_R(-)) \ar[r,"h^i_{S/R}(\mathcal{F}^e_R(-))"']
      & H^i(S\otimes_{R} \mathcal{F}^e_R(-))
    \end{tikzcd}
  \end{center}
  where the vertical arrows are isomorphisms. 
\end{corol}
In other words, if we identify $\mathcal{F}^e_S(S\otimes_R H^i(-))$ with $S\otimes_R H^i(\mathcal{F}^e_R(-))$ and $\mathcal{F}^e_S(H^i(S\otimes_R -))$ with
$H^i(S\otimes \mathcal{F}^e_R(-))$ using the vertical isomorphisms, then for
all $e\geq 0$, $h^i_{S/R}(\mathcal{F}^e_R(-))=\mathcal{F}^e_S(h^i_{S/R}(-))$. Concretely, if $C^\bullet = K^\bullet(\mathbf{\underline{f}};R)$ is the Koszul complex on some sequence of elements
$\mathbf{\underline{f}} = f_1,\cdots,f_t\in R$, and if $\mathbf{\underline{f}}^{[p^e]}:= f_1^{p^e},\cdots,f_t^{p^e}$, then a consequence of Corollary \ref{frob} is that the family of maps $\{h^i_e\}_{e=0}^\infty$ given by
$$
h^i_e:S\otimes_R H^i(\mathbf{\underline{f}}^{[p^e]};R)\to H^i(\mathbf{\underline{f}}^{[p^e]};S).
$$
  is determined by the data of just the $0$th map $h^i_0$ by taking Frobenius powers. The statement is more abstruse when applied directly to the \v{C}ech complex, since $\check{C}(\mathbf{\underline{f}};R)$ is canonically identified with $\mathcal{F}_R(\check{C}(\mathbf{\underline{f}};R))$. However, when we replace $R$ with a general $F$-finite $F_R$-module $\mathcal{M}$ in the sequel, it will simplify our main proofs if we work at the level of \v{C}ech cohomology.

\subsection{The natural map on local cohomology}
Throughout this section, $R$ and $S$ are regular rings of prime characteristic $p>0$. Recall that an $F_R$-module is a pair $(\mathcal{M},\theta)$ consisting of an $R$-module $\mathcal{M}$ and an isomorphism $\theta:\mathcal{M}\to\mathcal{F}_R(\mathcal{M})$. A morphism
of $F_R$-modules $(\mathcal{M},\theta)\to (\mathcal{N},\phi)$ is an $R$-linear map $h:\mathcal{M}\to\mathcal{N}$ such
that $\phi\circ h = \mathcal{F}_R(h)\circ \theta$. If $R\to S$ is a ring homomorphism and $(\mathcal{M},\theta)$
is an $F_R$-module, then $S\otimes_R \mathcal{M}$ is an $F_S$-module via the structure isomorphism
$\text{id}_S\otimes \theta:S\otimes_R \mathcal{M}\to S\otimes_R \mathcal{F}_R(\mathcal{M})$, where we identify the $S\otimes_R \mathcal{F}_R(\mathcal{M})$ with $\mathcal{F}_S(S\otimes_R \mathcal{M})$ in the canonical way.

Let $\mathbf{\underline{f}}=f_1,\cdots,f_t$ be a sequence of elements of $R$, let $C^\bullet:=\check{C}^\bullet(\mathbf{\underline{f}};R)$ be the \v{C}ech complex on $R$ associated with $\mathbf{\underline{f}}$. For shorthand, if $M$ is an $R$-module, denote by $C^\bullet_M:= C^\bullet\otimes_R M =\check{C}^\bullet(\mathbf{\underline{f}};M)$ the \v{C}ech complex on $M$ associated with $\mathbf{\underline{f}}$. The complex $\mathcal{F}_R(C^\bullet)$ is canonically identified with $C^\bullet$ itself, and likewise,
for any $R$-module $M$, $\mathcal{F}_R(C^\bullet_M)$ is canonically identified with $C^\bullet_{\mathcal{F}_R(M)}$. If $(\mathcal{M},\theta)$ is an $F_R$-module, then applying
$C^\bullet\otimes_R -$ to $\theta:\mathcal{M}\to \mathcal{F}_R(\mathcal{M})$ gives an isomorphism of complexes
$\Theta:C^\bullet_{\mathcal{M}}\to \mathcal{F}_R(C^\bullet_{\mathcal{M}})$. Using this isomorphism, $H^i_I(\mathcal{M})$ can naturally be made into an $F_R$-module for any $i$: its structure isomorphism $\Psi: H^i_I(\mathcal{M})\to \mathcal{F}_R(H^i_I(\mathcal{M}))$ is given by the composition
\vspace{-1.2em}
\begin{center}
  \begin{equation}\label{decomp}
  \begin{tikzcd}
    H^i(C^\bullet_{\mathcal{M}})\ar[r,"H^i(\Theta)"] &
    H^i(\mathcal{F}_R(C^\bullet_{\mathcal{M}}))\ar[rr,"h^i_{F_R}(C^\bullet_{\mathcal{M}})^{-1}"] &&
    \mathcal{F}_R(H^i(C^\bullet_{\mathcal{M}}))
  \end{tikzcd}\tag{$\star\star$}
  \end{equation}
\end{center}

We now prove our main compatibility result.
\begin{theor}\label{morph}
Let $R\to S$ be a homomorphism between two regular rings of prime characteristic $p>0$, fix an ideal $I\subseteq R$ and an index $i\geq 0$, and let $(\mathcal{M},\theta)$ be an
$F_R$-module. The natural map
$$
S\otimes_R H^i_I(\mathcal{M})\to H^i_I(S\otimes_R \mathcal{M})
$$
is a morphism of $F_S$-modules. 
\end{theor}
\begin{proof}
  Let $C^\bullet_{\mathcal{M}} = \check{C}^\bullet(\mathbf{\underline{f}};\mathcal{M})$
  be the \v{C}ech complex on $\mathcal{M}$ associated with a sequence of elements $\mathbf{\underline{f}}=f_1,\cdots,f_t$ generating $I$. It is enough to show that the diagram
\begin{center}
\begin{tikzcd}[column sep = huge]
  \mathcal{F}_S(S\otimes_R H^i(C^\bullet_{\mathcal{M}}))\ar[r,"\mathcal{F}_S(h^i_{S/R}(C^\bullet_{\mathcal{M}}))"]\ar[d,"\text{\rotatebox{90}{$\sim$}}"'] & \mathcal{F}_S(H^i(S\otimes_R C^\bullet_{\mathcal{M}}))\ar[d,"\text{\rotatebox{90}{$\sim$}}"']\\
  S\otimes_R H^i(C^\bullet_{\mathcal{M}})\ar[r,"h^i_{S/R}(C^\bullet_{\mathcal{M}})"] & H^i(S\otimes_RC^\bullet_{\mathcal{M}})
\end{tikzcd}
\end{center}
commutes, where the vertical arrows are the inverses of the structure morphisms of
$S\otimes_R H^i_I(\mathcal{M})$ and $H^i_I(S\otimes_R \mathcal{M})$ as $F_S$-modules, respectively. Let $\Theta:C^\bullet_{\mathcal{M}}\to\mathcal{F}(C^\bullet_{\mathcal{M}})$ denote the isomorphism of complexes induced by $\theta$. Using
the decomposition \eqref{decomp} of the structure isomorphism of $H^i_I(\mathcal{M})$, the stated result is equivalent to showing that the following diagram commutes.
\begin{center}
\begin{tikzcd}
  \mathcal{F}_S(S\otimes_R H^i(C^\bullet_{\mathcal{M}}))
  \ar[d,no head, shift right = 0.1em]\ar[d,no head,shift left=0.1em]
  \ar[rr,"\mathcal{F}_S(h^i_{S/R}(C^\bullet_{\mathcal{M}}))"]
  && \mathcal{F}_S(H^i(S\otimes_R C^\bullet_{\mathcal{M}}))
  \ar[d,"h^i_{F_S}(S\otimes_R C^\bullet_{\mathcal{M}})","\text{\rotatebox{90}{$\sim$}}"']\\

  S\otimes_R \mathcal{F}_R(H^i(C^\bullet_{\mathcal{M}}))
  \ar[d,"\text{id}_S\otimes h^i_{F_R}(C^\bullet_{\mathcal{M}})"',"\text{\rotatebox{90}{$\sim$}}"]
  && H^i(\mathcal{F}_S(S\otimes_R C^\bullet_{\mathcal{M}}))
  \ar[d,no head, shift right = 0.1em]\ar[d,no head,shift left=0.1em]\\

  S\otimes_R H^i(\mathcal{F}_R(C^\bullet_{\mathcal{M}}))
  \ar[d,"\text{id}_S\otimes H^i(\Theta)^{-1}"',"\text{\rotatebox{90}{$\sim$}}"]
  \ar[rr,"h^i_{S/R}(\mathcal{F}_R(C^\bullet_{\mathcal{M}}))"]
  && H^i(S\otimes_R\mathcal{F}_R(C^\bullet_{\mathcal{M}}))
  \ar[d,"H^i(\text{id}_S\otimes \Theta)^{-1}","\text{\rotatebox{90}{$\sim$}}"']\\
  
  S\otimes_R H^i(C^\bullet_{\mathcal{M}})  
  \ar[rr,"h^i_{S/R}(C^\bullet_{\mathcal{M}})"]
  && H^i(S\otimes_R C^\bullet_{\mathcal{M}})
\end{tikzcd}
\end{center}
The commutativity of the rectangle of maps in the top three rows is precisely the content of
Proposition \ref{simflat} (see specifically Corollary \ref{frob}) applied to the complex $C^\bullet_{\mathcal{M}}$. The square of maps in the bottom two rows is induced
from the diagram that results from applying $H^i(-)$ to
\begin{center}
  \begin{tikzcd}
  \mathcal{F}_R(C^\bullet_{\mathcal{M}})
  \ar[d,"\Theta^{-1}"',"\text{\rotatebox{90}{$\sim$}}"]
  \ar[r,"\text{nat}"]
  & S\otimes_R\mathcal{F}_R(C^\bullet_{\mathcal{M}})
  \ar[d,"(\text{id}_S\otimes \Theta)^{-1}","\text{\rotatebox{90}{$\sim$}}"']\\
  
  C^\bullet_{\mathcal{M}}  
  \ar[r,"\text{nat}"]
  &S\otimes_R C^\bullet_{\mathcal{M}}

  \end{tikzcd}
\end{center}
Recall that $C^\bullet_{\mathcal{M}}=C^\bullet\otimes_R \mathcal{M}$ and $\mathcal{F}_R(C^\bullet_{\mathcal{M}}) = C^\bullet \otimes_R \mathcal{F}_R(\mathcal{M})$, where $C^\bullet =\check{C}^\bullet(\mathbf{\underline{f}};R)$, so that the above
diagram is $C^\bullet\otimes_R-$ applied to
\begin{center}
  \begin{tikzcd}
  \mathcal{F}_R({\mathcal{M}})
  \ar[d,"\Theta^{-1}"',"\text{\rotatebox{90}{$\sim$}}"]
  \ar[r,"\text{nat}"]
  & S\otimes_R\mathcal{F}_R(\mathcal{M})
  \ar[d,"(\text{id}_S\otimes \Theta)^{-1}","\text{\rotatebox{90}{$\sim$}}"']\\
  
  \mathcal{M}  
  \ar[r,"\text{nat}"]
  &S\otimes_R \mathcal{M}
  \end{tikzcd}
\end{center}
This final diagram obviously commutes.
\end{proof}

\begin{corol}\label{quot}
  Let $R$ be a regular ring of prime characteristic $p>0$, let $J\subseteq R$ be an ideal such that $R/J$ is regular, and let $\mathcal{M}$ be an $F_R$-module. For any ideal $I\subseteq R$ and any $i\geq 0$, the natural map $H^i_I(\mathcal{M})/J H^i_I(\mathcal{M}) \to H^i_I(\mathcal{M}/J\mathcal{M})$
  is an $F_{R/J}$-module morphism.
\end{corol}

\begin{theor}\label{charpj}
  Let $R$ be a regular ring of prime characteristic $p>0$, let $J\subseteq R$ be an ideal such that $R/J$ is regular, and let $\mathcal{M}$ be an $F$-finite $F_R$-module (e.g. $\mathcal{M}=R$). For any ideal $I\subseteq R$ and any $i\geq 0$, $$\text{Coker}\left(H^i_I(\mathcal{M}) \to H^i_I(\mathcal{M}/J\mathcal{M})\right)$$
  is an $F$-finite $F_{R/J}$-module, and hence, has finitely many associated primes. The module
  $H^i_I(J\mathcal{M})$ also has finitely many associated primes.
\end{theor}
\begin{proof}
  Note that $H^i_I(\mathcal{M})$ is an $F$-finite $F_R$-module, and $H^i_I(\mathcal{M}/J\mathcal{M})$ is an $F$-finite $F_{R/J}$-module \cite[Proposition 2.10]{lyufmod}. The map $H^i_I(\mathcal{M}) \to H^i_I(\mathcal{M}/J\mathcal{M})$ factors through the natural map $H^i_I(\mathcal{M}) \to H^i_I(\mathcal{M})/J H^i_I(\mathcal{M})$, and since $R\twoheadrightarrow R/J$ is surjective, the images of $H^i_I(\mathcal{M})$ and $H^i_I(\mathcal{M})/J H^i_I(\mathcal{M})$ inside $H^i_I(\mathcal{M}/J\mathcal{M})$ are equal. Thus,
  $$\text{Coker}\left(H^i_I(\mathcal{M}) \to H^i_I(\mathcal{M}/J\mathcal{M})\right) =
  \text{Coker}\left(H^i_I(\mathcal{M})/JH^i_I(\mathcal{M}) \to H^i_I(\mathcal{M}/J\mathcal{M})\right)$$
  where the latter is the cokernel of a morphism of $F$-finite $F_{R/J}$-modules by Corollary \ref{quot}. It is therefore itself an $F$-finite $F_{R/J}$-module \cite[Theorem 2.8]{lyufmod}, and accordingly, has a finite set of associated primes \cite[Theorem 2.12(a)]{lyufmod}.
  Regarding the claim about the associated primes of $H^i_I(J\mathcal{M})$, apply
  $\Gamma_I(-)$ to $0\to J\mathcal{M}\to\mathcal{M}\to\mathcal{M}/J\mathcal{M}\to 0$ to obtain the exact sequence
\begin{center}{
\begin{tikzpicture}[descr/.style={fill=white,inner sep=1.5pt}]
        \matrix (m) [
            matrix of math nodes,
            row sep=1em,
            column sep=2.5em,
            text height=1.5ex, text depth=0.25ex
        ]
        { & \cdots & H^{i-1}_I(\mathcal{M}) & H^{i-1}_I(\mathcal{M}/J\mathcal{M}) \\
          & H^{i}_I(J\mathcal{M}) & H^{i}_I(\mathcal{M}) & \cdots \\
        };

        \path[overlay,->, font=\scriptsize,>=latex]
        (m-1-2) edge (m-1-3)
        (m-1-3) edge (m-1-4)
        (m-1-4) edge[out=345,in=165] (m-2-2)
        (m-2-2) edge (m-2-3)
        (m-2-3) edge (m-2-4)
;
\end{tikzpicture}}
\end{center}
  We therefore have a short exact sequence
  $$
  0\to \text{Coker}\left(H^{i-1}_I(\mathcal{M}) \to H^{i-1}_I(\mathcal{M}/J\mathcal{M})\right) \to H^i_I(J\mathcal{M})\to N\to 0
  $$
  for some submodule $N\subseteq H^i_I(\mathcal{M})$, and the stated result now follows at once. 
\end{proof}

Let  $(R,\mathfrak{m},K)$ be a regular local ring. Recall that a parameter ideal $J\subseteq R$ is called {\it regular} if it is generated by an $R$-regular sequence whose images in $\mathfrak{m}/\mathfrak{m}^2$ are linearly independent over $K$. Every ideal $J$ such that $R/J$ is regular has this form. If $R$ is complete and contains a field, then by then Cohen Structure Theorem, all examples of regular parameter ideals are isomorphic to an example of the form $R=K[[x_1,\cdots,x_m,z_1,\cdots,z_n]]$ and $J=(x_1,\cdots,x_m)R$ for some $m,n\geq 0$. 

\begin{corol}
  Let $(R,\mathfrak{m},K)$ be a regular local ring containing a field of prime characteristic $p>0$ and $J\subseteq R$ be a regular parameter ideal. For any ideal $I\subseteq R$ and any $i\geq 0$, the module
  $H^i_I(J)$ has finitely many associated primes.
\end{corol}

\section*{Acknowledgements}
\noindent The author would like to thank her advisor, Mel Hochster, for his invaluable comments, feedback, and suggested improvements to this project, and for his support throughout her doctoral work more generally.

\bibliographystyle{alpha}
\bibliography{LocalCohomologyParameterIdeal}

\end{document}